\documentclass[graybox]{svmult}

\usepackage{mathptmx}       
\usepackage{helvet}         
\usepackage{courier}        
\usepackage{type1cm}        
\usepackage{makeidx}         
\usepackage{graphicx}        
\usepackage{multicol}        
\usepackage[bottom]{footmisc}

\usepackage{amssymb}

\usepackage{tikz}

\makeindex             

\spnewtheorem{assumption}[lemma]{Assumption}{\bfseries}{\itshape}
\spnewtheorem{mytheorem}[lemma]{Theorem}{\bfseries}{\itshape}
\newcommand{\tfrac}[2]{\textstyle\frac{#1}{#2}}

\begin{document}

\title*{Stability preserving approximations of a semilinear hyperbolic gas transport model}
\titlerunning{Stability preserving approximations of a semilinear hyperbolic gas transport model}
\author{Herbert Egger, Thomas Kugler and Bj\"orn Liljegren-Sailer}
\institute{Herbert Egger, Thomas Kugler \at Technische Universit\"at Darmstadt, Germany, \email{\{egger,kugler\}@mathematik.tu-darmstadt.de} 
\and Bj\"orn Liljegren-Sailer \at Universit\"at Trier, Germany, \email{bjoern.sailer@uni-trier.de}}

\maketitle

\def\myabstract{We consider the discretization of a semilinear damped wave equation arising, for instance, in the modeling of gas transport in pipeline networks. 
For time invariant boundary data, the solutions of the problem are shown to converge exponentially fast to steady states. We further prove that this decay behavior 
is inherited uniformly by a class of Galerkin approximations, including finite element, spectral, and structure preserving model reduction methods.  
These theoretical findings are illustrated by numerical tests.}

\abstract*{\myabstract}

\abstract{\myabstract}

\section{Introduction} \label{sec:1}

The propagation of pressure waves through a network of pipes can be described by a semilinear hyperbolic system on each pipe together with appropriate coupling conditions \cite{BrouwerGasserHerty11,Mugnolo14}. Due to friction at the pipe walls, the kinetic energy of the gas flow gets damped resulting in a dissipative behavior and, as a consequence, the system relaxes to steady states exponentially fast; see \cite{GattiPata06}. 
While structure preserving model reduction methods \cite{Gugercin12} allow to guarantee the dissipative nature also after discretization, 
the rates of the exponential decay in the discretized models may in general degenerate with the discretization parameter; see e.g. \cite{Tebou03}. 

In this work we extend our previous results \cite{EggerKuglerEtAl18} to problems with nonlinear damping and make the following contributions: 
First, we prove the exponential decay for the infinite dimensional problem in a form that can be extended to pipe networks. 
Second, we analyze a class of Galerkin discretizations which inherit the exponential decay behavior uniformly in the discretization parameter.
%

\section{Analytical results} \label{sec:2}

We consider the semilinear instationary wave propagation problem
\begin{eqnarray}
\partial_{t} p(x,t) + \partial_{x} m(x,t) &=& \bar{f}(x),  \qquad x\in(0,1),\,t\in[0,T],\label{sys:inst1}\\
\partial_{t} m(x,t) + \partial_{x} p(x,t) + d( m(x,t)) &=& \bar{g}(x), \qquad x\in(0,1),\,t\in[0,T],\label{sys:inst2}\\
 p(x,t) &=& \bar{h}(x), \qquad x\in\{0,1\},\,t\in[0,T],\label{sys:inst3}
\end{eqnarray}
with nonlinear damping function $d$ satisfying the following assumptions.
\begin{assumption} \label{ass:damp}
$d \in C^1(\mathbb{R})$ with $d(0) = 0$, $d'(m)>d_0$, and $|d'(m)|\le d_1 + d_2|m|^p$ for some constants 
$d_0>0$ and $d_1,d_2,p\ge0$.
\end{assumption}
These conditions allow us to prove the well-posedness of the above problem.
As a preparatory step, let us consider corresponding stationary problems of the form 
\begin{eqnarray}
\partial_{x}\tilde m(x) &=& \tilde f(x),  \qquad x\in(0,1),\label{sys:st1}\\
\partial_{x}\tilde p(x) + \tilde d(\tilde m(x)) &=& \tilde g(x), \qquad x\in(0,1),\label{sys:st2}\\
\tilde p(x) &=& \tilde h(x), \qquad x\in\{0,1\}.\label{sys:st3}
\end{eqnarray}
Note that solutions of (\ref{sys:st1})--(\ref{sys:st3}) with $(\tilde d,\tilde f,\tilde g,\tilde h) = (d,\bar{f},\bar{g},\bar{h})$ are steady states $(\bar{p},\bar{m})$
for the system (\ref{sys:inst1})--(\ref{sys:inst3}). 
Using Assumption~\ref{ass:damp} and results about nonlinear variational problems under constraints \cite[Proposition 2.3]{Scheurer77}, we obtain the following.
\begin{lemma} \label{lem:stat}
Let Assumption~\ref{ass:damp} hold. Then for any $\tilde f,\tilde g\in L^2(0,1)$ and $\tilde h\in\mathbb{R}^2$ 
the system (\ref{sys:st1})--(\ref{sys:st3}) has a unique solution $(\tilde{p},\tilde{m})\in H^1(0,1) \times H^1(0,1)$ and there exists a constant $c>0$ 
independent of $\tilde d$ and of $\tilde f, \tilde g, \tilde h$, such that
 \begin{eqnarray*}
 \|\tilde{m}\|_{H^1} &\le& \frac{c}{d_0}(\|\tilde g\|_{L^2} + |\tilde h|_1 + d_1\|\tilde f\|_{L^2} + d_2\|\tilde f\|_{L^2}^{p+1}) + c\|\tilde f\|_{L^2} :=M \\
  \|\tilde{p}\|_{H^1} &\le& c\:(\|\tilde g\|_{L^2} + |\tilde h|_1 + d_1M + d_2M^{p+1}).
 \end{eqnarray*}
\end{lemma}
Let us now return to the instationary problem.  
Using the previous result and energy estimates, we can show the following a-priori bounds.
\begin{lemma} \label{lem:apriori}
Let $( p, m)$ be a smooth solution of (\ref{sys:inst1})--(\ref{sys:inst3}). 
Then 
 \begin{eqnarray*}
  \|\partial_{t} p(t)\|_{L^2} + \|\partial_{t} m(t)\|_{L^2} + \| m(t)\|_{H^1} \le c\:(\|\bar{f}\|_{L^2},\|\bar{g}\|_{L^2},|\bar{h}|_1,\| p(0)\|_{H^1},\| m(0)\|_{H^1})
 \end{eqnarray*}
with a constant $c$ depending only on $\bar{f},\bar{g},\bar{h}, p(0), m(0)$ but not on times $t$ and $T$.
\end{lemma}
Here and below, we interpret $p$ and $m$ as functions of time with values in a Hilbert space, and write $p(t)$ and $m(t)$ 
for the corresponding functions of $x$ at time $t$. 
\begin{proof} \smartqed 
Subtracting equations (\ref{sys:st1})--(\ref{sys:st3}) for $(\tilde  f,\tilde g,\tilde h) = (\bar{f},\bar{g},\bar{h})$ from (\ref{sys:inst1})--(\ref{sys:inst3}) yields a problem of the form (\ref{sys:inst1})--(\ref{sys:inst3}) for the functions $( p(t)-\bar{p}, m(t)-\bar{m})$ with $\bar{f},\bar{g},\bar{h} = 0$ and damping term 
$d( m)$ replaced by $d( m(t)) - d(\bar{m})$. 
By testing this problem with the functions $( p(t)-\bar{p}, m(t)-\bar{m})$ and noting that $(d( m)-d(\bar{m}), m-\bar{m}) \ge 0$ due to Assumption~\ref{ass:damp}, 
one can see that
\begin{eqnarray*}
 \| p(t)-\bar{p}\|_{L^2}^2 + \| m(t)-\bar{m}\|_{L^2}^2 &\le& \| p_0-\bar{p}\|_{L^2}^2 + \| m_0-\bar{m}\|_{L^2}^2.
\end{eqnarray*}
Differentiation of (\ref{sys:inst1})--(\ref{sys:inst3}) with respect to time, testing with $(\partial_{t} p(t),\partial_{t} m(t))$, and using that $d'(m)>0$ 
by Assumption~\ref{ass:damp}, further shows that 
\begin{eqnarray*}
 \|\partial_{t} p(t)\|_{L^2}^2 + \|\partial_{t} m(t)\|_{L^2}^2 &\le& \|\partial_{t} p(0)\|_{L^2}^2 + \|\partial_{t} m(0)\|_{L^2}^2.
\end{eqnarray*}
The right-hand side in this estimate can be bounded using (\ref{sys:inst1})--(\ref{sys:inst3}) for $t=0$. 
Then the splitting $\| m(t)\| \le \| m(t)-\bar{m}\| + \|\bar{m}\|$ and $\|\partial_{x} m(t)-\partial_{x}\bar{m}\| = \|\partial_{t} p(t)\|$ together with the bounds of Lemma~\ref{lem:stat} and the previous estimates implies the result. \qed
\end{proof}
We are now in the position to show well-posedness of the instationary problem. 
\begin{lemma} \label{lem:instat}
Let Assumption~\ref{ass:damp} hold. Then for any $\bar{f},\bar{g}\in L^2(0,1)$, any $\bar{h}\in\mathbb{R}^2$, and any $ p_0, m_0\in H^1(0,1)$ 
there exists a unique solution $( p, m)\in C(0,T;H^1\times H^1)\cap C^1(0,T;L^2\times L^2)$ of the system (\ref{sys:inst1})--(\ref{sys:inst3}) 
with initial value $ p(0)= p_0$ and $  m(0)= m_0$.
\end{lemma}
\begin{proof} \smartqed
By Assumption~\ref{ass:damp}, the nonlinear damping term $d( m)$ in equation (\ref{sys:inst2}) is locally Lipschitz continuous, and existence of a 
unique solution $( p, m)$ local in time thus follows by semigroup theory; cf. \cite[Theorem 6.1.4]{Pazy83}.
The uniform a-priori estimates of Lemma \ref{lem:apriori} allow to extend the solution globally in time.\qed
\end{proof}
We can now state our first main result, i.e. the exponential decay of the energies
\begin{eqnarray*}
 E( q, v)&:=& \tfrac{1}{2}\| q\|_{L^2}^2 +  \tfrac{1}{2}\| v\|_{L^2}^2
\end{eqnarray*}
for the two choices $( q, v) = ( p(t)-\bar{p}, m(t)-\bar{m})$ and $( q, v) = (\partial_{t} p(t),\partial_{t} m(t))$.
\begin{mytheorem}
 Let $( p, m)$ be a solution of (\ref{sys:inst1})--(\ref{sys:inst3}) provided by Lemma \ref{lem:instat}. Then 
 \begin{eqnarray*}
  E( p(t)-\bar{p}, m(t)-\bar{m}) \le ce^{-\gamma t}\quad\mbox{and}\quad E(\partial_{t} p(t),\partial_{t} m(t)) \le c'e^{-\gamma t},
 \end{eqnarray*}
 for $0\le t\le T$ with $c,c',\gamma>0$ only depending on $\|\bar{f}\|_{L^2},\|\bar{g}\|_{L^2},|\bar{h}|_2,\| p_0\|_{H^1},\| m_0\|_{H^1}$.
\end{mytheorem}
The proof follows in the same way as that of Theorem \ref{thm:decayh} given below, and is therefore omitted. 
Similar results can also be found in \cite{BabinVishik83,GattiPata06,Zuazua88}.

\vspace*{-1em}
\section{Galerkin discretization in space} \label{sec:3}

Let $Q_h\subset L^2(0,1)$ and $V_h\subset H^1(0,1)$ and consider the following Galerkin approximation of 
the stationary problem (\ref{sys:st1})--(\ref{sys:st3}): Find $(\tilde{p}_h,\tilde{m}_h)\in Q_h\times V_h$ such that
\begin{eqnarray}
 (\partial_{x}\tilde{m}_h,{q}_h) &=& (\tilde f,{q}_h), \label{sys:sth1}\\
 - (\tilde{p}_h,\partial_{x}{v}_h) + (\tilde d(\tilde{m}_h),{v}_h) &=& (\tilde g,{v}_h) - \tilde h{v}_h|_0^1, \label{sys:sth2}
\end{eqnarray}
for all ${q}_h\in Q_h$ and ${v}_h\in V_h$. For convenience we write $(\cdot,\cdot):=(\cdot,\cdot)_{L^2}$ in the sequel.
We will assume that the spaces $Q_h$, $V_h$ satisfy the following compatibility conditions.

\pagebreak

\begin{assumption} \label{ass:compat}
$Q_h\subset L^2(0,1)$ and $V_h\subset H^1(0,1)$ are finite dimensional and 
\begin{eqnarray}
 Q_h = \partial_{x} V_h\quad\mbox{and}\quad \mathop{ker}(\partial_{x})\subset V_h.
\end{eqnarray}
\end{assumption}
Well-posedness of the discretized stationary problem (\ref{sys:sth1})--(\ref{sys:sth2}) now follows with the same arguments as used in Lemma~\ref{lem:stat} for the analysis on the continuous level.
\begin{lemma} \label{lem:stath}
Let Assumptions~\ref{ass:damp} and \ref{ass:compat} hold. 
Then for any $\tilde f,\tilde g\in L^2(0,1)$ and $\tilde h\in\mathbb{R}^2$ there exists a unique solution 
$(\tilde p_h,\tilde m_h)\in Q_h\times V_h$ of the system (\ref{sys:sth1})--(\ref{sys:sth2}) and a constant 
$c>0$ independent of $\tilde d$, of $\tilde f, \tilde g, \tilde h$ and of the space $Q_h,V_h$, such that
 \begin{eqnarray*}
 \|\tilde m_h\|_{H^1} &\le&  \frac{c}{d_0}(\|\tilde g\|_{L^2} + |\tilde h|_1 + d_1\|\tilde f\|_{L^2} + d_2\|\tilde f\|_{L^2}^{p+1}) + c\|\tilde f\|_{L^2} :=M \\
  \|\tilde p_h\|_{H^1} &\le& c\:(\|\tilde g\|_{L^2} + |\tilde h|_1 + d_1M + d_2M^{p+1}).
 \end{eqnarray*}
\end{lemma}
The corresponding discretization of the instationary problem (\ref{sys:inst1})--(\ref{sys:inst3}) reads as follows:  
Find $( p_h, m_h)\in H^1(0,T;Q_h\times V_h)$ such that 
\begin{eqnarray}
(\partial_{t} p_h(t),{q}_h) + (\partial_{x} m_h(t),{q}_h) &=& (\bar{f},{q}_h), \label{sys:insth1}\\
(\partial_{t} m_h(t),{v}_h) - ( p_h(t),\partial_{x}{v}_h) + (d( m_h(t)),{v}_h) &=& (\bar{g},{v}_h) - \bar{h}{v}_h|_0^1, \label{sys:insth2}
\end{eqnarray}
for all ${q}_h\in Q_h$ and ${v}_h\in V_h$, and for $0\le t\le T$. In addition, we require that
\begin{eqnarray}
  p_h(0) =  p_{h,0}\quad\mbox{and}\quad m_h(0) =  m_{h,0}, \label{sys:insth3}
\end{eqnarray}
where $( p_{h,0}, m_{h,0})$ solves problem (\ref{sys:sth1})--(\ref{sys:sth2}) with $(\tilde f,{q}_h) = (\partial_{x} m_0,{q}_h)$ and   $(\tilde g,{v}_h) = (d( m_0),{v}_h)-( p_0,\partial_{x}{v}_h)$. 
By Lemma~\ref{lem:stath}, $ p_{h,0}, m_{h,0}$ and $\partial_{t} p_{h,0},\partial_{t} m_{h,0}$ can be bounded in terms of the data of the continuous problem. 
%
%
In order to prove the existence of a global solution, we proceed similarly as on the continuous level. 
We denote by $(\bar{p}_h,\bar{m}_h)$ the steady states of the system (\ref{sys:insth1})--(\ref{sys:insth2}), which correspond to the solution of 
(\ref{sys:sth1})--(\ref{sys:sth2}) with $(\tilde d,\tilde f,\tilde g,\tilde h) = (d,\bar{f},\bar{g},\bar{h})$, and obtain the following a-priori bounds.
\begin{lemma} \label{lem:apriorih}
Any solution $( p_h, m_h)\in H^1(0,T;Q_h\times V_h)$ of (\ref{sys:insth1})--(\ref{sys:insth3}) with initial values $ p_{h,0}$ and $ m_{h,0}$ as described above, satisfies
 \begin{eqnarray*}
  \|\partial_{t} p_h(t)\|_{L^2} + \|\partial_{t} m_h(t)\|_{L^2} + \| m_h(t)\|_{H^1} \le c\:(\|\bar{f}\|_{L^2},\|\bar{g}\|_{L^2},|\bar{h}|_1,\| p_0\|_{H^1},\| m_0\|_{H^1})
 \end{eqnarray*}
with a constant $c>0$ depending only on $\bar{f},\bar{g},\bar{h}, p_0, m_0$ but not on $t$, $T$, or $Q_h$, $V_h$.
\end{lemma}
By the Picard-Lindel\"of theorem, one then obtains the existence of a unique solution.
\begin{lemma} \label{lem:wph}
Let the conditions of Lemma~\ref{lem:instat} and Assumption~\ref{ass:compat} hold. 
Then there exists a unique solution $( p_h, m_h)\in H^1(0,T;Q_h\times V_h)$ of problem (\ref{sys:insth1})--(\ref{sys:insth3}).
\end{lemma}
We are now in the position to prove the main result of our paper.
\begin{mytheorem} \label{thm:decayh}
 Under the assumptions of Lemma~\ref{lem:stath} and \ref{lem:wph}, there holds
 \begin{eqnarray*}
  E( p_h(t)-\bar{p}_h, m_h(t)-\bar{m}_h) \le ce^{-\gamma t}\quad\mbox{and}\quad E(\partial_{t} p_h(t),\partial_{t} m_h(t)) \le c'e^{-\gamma t},
 \end{eqnarray*}
for all $0\le t\le T$ with constants $c,c',\gamma>0$ depending only on the data. 
\end{mytheorem}
In particular, the estimate is independent of $T$ and the choice of the spaces $Q_h,V_h$.
\begin{proof} \smartqed 
For any $t\in[0,T]$ the difference $(\tilde p_h,\tilde m_h):=( p_h(t)-\bar{p}_h, m_h(t)-\bar{m}_h)$ satisfies (\ref{sys:sth1})--(\ref{sys:sth2}) with $\tilde f = \partial_{t} p_h(t)$, $\tilde g = \partial_{t} m_h(t)$, and damping  $\tilde d(\tilde m_h) := d(\tilde m_h + \bar{m}_h)-d(\bar{m}_h)$. From Assumption~\ref{ass:damp}, one can deduce that $\tilde d'(m) \ge d_0 > 0 $ and
\[
|\tilde d'(m)| = |d'(m+\bar{m}_h)| \le d_1 + d_2 |m+\bar{m}_h|^p \le \tilde d_1 + \tilde d_2 |m|^p, 
\]
for some constants $\tilde d_1$, $\tilde d_2$ depending only on $d_1$, $d_2$, $p$ and the norm of the steady state $\bar{m}_h$, which is 
bounded uniformly by Lemma~\ref{lem:stath} in terms of the data.
Therefore, the a-priori estimates of Lemma~\ref{lem:stath} apply and we can further estimate the terms $\|\tilde f\|_{L^2}^p$ and $M^p$ 
appearing in the estimate of Lemma~\ref{lem:stath} by Lemma \ref{lem:apriorih}. As a consequence
\begin{eqnarray} \label{eq:mdiffh}
  \| p_h(t)-\bar{p}_h\|_{L^2} + \| m_h(t)-\bar{m}_h\|_{L^2} \le c\:(\|\partial_{t} p_h(t)\|_{L^2} + \|\partial_{t} m_h(t)\|_{L^2})
\end{eqnarray}
with some constant $c$ independent of $t$, $T$, and of the spaces $Q_h,V_h$. 
Let us define a modified energy $E_{h,\varepsilon}^1 := E(\partial_{t} p_h,\partial_{t} m_h) + \varepsilon(\partial_{t} m_h, m_h-\bar{m}_h)_{L^2}$ and note that 
\begin{equation} \label{eq:equivE}
 \tfrac{1}{2} E(\partial_{t} p_h,\partial_{t} m_h) \le E_{h,\varepsilon}^1  \le \tfrac{3}{2} E(\partial_{t} p_h,\partial_{t} m_h) 
\end{equation}
for all parameters $0 \le \varepsilon \le\varepsilon^*$ sufficiently small, i.e., the two energies are equivalent. 
From (\ref{sys:insth1})--(\ref{sys:insth2}), one can further deduce that 
\begin{eqnarray*}
 \frac{d}{dt}E_{h,\varepsilon}^1 &=& \frac{d}{dt}E(\partial_{t} p_h,\partial_{t} m_h) + \varepsilon\|\partial_{t} m_h\|_{L^2}^2 + \varepsilon(\partial_{tt} m_h, m_h-\bar{m}_h)_{L^2} \\
 &\le& -(d_0-\varepsilon)\|\partial_{t} m_h\|_{L^2}^2 + \varepsilon(\partial_{tt} m_h, m_h-\bar{m}_h)_{L^2}.
\end{eqnarray*}
The second term in this estimate can be bounded by
\begin{eqnarray*}
 (\partial_{tt} m_h, m_h-\bar{m}_h)_{L^2} &=& (\partial_{t} p_h,\partial_{x}( m_h-\bar{m}_h))_{L^2} - (d'( m_h)\partial_{t} m_h, m_h-\bar{m}_h)_{L^2} \\
 &\le& -\|\partial_{t} p_h\|_{L^2}^2 +  c\|\partial_{t} m_h\|_{L^2}\| m_h-\bar{m}_h\|_{L^2} \\
 &\le& -\tfrac{1}{2}\|\partial_{t} p_h\|_{L^2}^2 + \tilde c\|\partial_{t} m_h\|_{L^2}^2,
\end{eqnarray*}
where the global a-priori bounds in Lemma \ref{lem:apriorih}, equation (\ref{eq:mdiffh}), and the assumptions on $d$ were used. 
By choosing $\varepsilon^*>0$ sufficiently small, we can conclude that
\begin{eqnarray*}
 \frac{d}{dt}E_{h,\varepsilon}^1 \le -\varepsilon E(\partial_t p_h,\partial_t m_h) \le -\tfrac{2\varepsilon}{3} E_{h,\varepsilon}^1, \qquad \mbox{for all } 0 < \varepsilon \le \varepsilon^*.
\end{eqnarray*}
By integration in time, this yields $E_{h,\varepsilon}^1(t) \le e^{-\frac{2\varepsilon}{3}} E_{h,\varepsilon}^1(0)$, 
and using the equivalence  of the two energies (\ref{eq:equivE}), we obtain the second estimate of the theorem.
With the help of inequality (\ref{eq:mdiffh}), we also obtain the first estimate. \qed
\end{proof}
\begin{remark} \label{rem:approx}
In the next section, we will make use of the following simple observation:
Let $(\cdot,\cdot)_h$ be a semi inner product which is equivalent to $(\cdot,\cdot)_{L^2}$ on $V_h$, i.e., 
\begin{equation} \label{eq:equiv}
 \tfrac{1}{2} \|{v}_h\|_{L^2} \le \|{v}_h\|_h \le \tfrac{3}{2}\|{v}_h\|_{L^2}\quad\mbox{for all }{v}_h\in V_h.
\end{equation}
Then the assertions of Theorem~\ref{thm:decayh} remain valid when replacing  $(\partial_{t} m_h(t),{v}_h)$ 
and $(d( m_h(t)),{v}_h)$ in problem (\ref{sys:insth1})--(\ref{sys:insth3}) by the approximations $(\partial_{t} m_h(t),{v}_h)_h$ and $(d( m_h(t)),{v}_h)_h$, 
which can be verified by a close inspection of the previous proof. 
This modification may substantially simplify the numerical solution.
\end{remark}

\section{Approximation schemes} \label{sec:4}

After a basis is chosen for $Q_h$ and $V_h$, the discretized system (\ref{sys:insth1})--(\ref{sys:insth2}) reads
\begin{eqnarray}
\mathbf{M}_p \partial_{t}\mathbf{p}(t) + \mathbf{G} \mathbf{m}(t) &=& \mathbf{f}(t), \label{sys:instha1}\\
\mathbf{M}_m \partial_{t}\mathbf{m}(t) - \mathbf{G}^T\mathbf{p}(t) + \mathbf{D}(\mathbf{m}(t))\mathbf{m}(t) &=& \mathbf{g}(t) - \mathbf{B} \mathbf{h}(t). \label{sys:instha2}
\end{eqnarray}
Here $\mathbf{p}, \mathbf{m}$ are the coordinate vectors for the functions $ p_h$, $ m_h$.
Following Remark~\ref{rem:approx}, we define quadrature points $\xi_n$ and weights $\omega_n$, and we set
\begin{eqnarray} \label{eq:quadratur}
( v,\tilde  v)_h := \sum_{n=0}^N \omega_n  v(\xi_n) \tilde  v(\xi_n), \qquad \mbox{ for } v, \tilde  v \in H_1.
\end{eqnarray}
We now discuss some typical choices for the subspaces $Q_h$, $V_h$ for method (\ref{sys:insth1})--(\ref{sys:insth2}). 
\begin{example}[Finite element method] \label{ex:hfem}
Let $T_h$ be a uniform mesh with nodes $x_n = nh$, $h=1/N$, and let $P_p(T_h)$ be the space of piecewise polynomials of order $p$. 
We set $Q_h = P_0(T_h)$ and $V_h = P_1(T_h)\cap H^1$ and note that Assumption \ref{ass:compat} is satisfied.
We further choose $\xi_n = x_n$ and $\omega_0= \omega_N = h/2$ and $\omega_n = h$ for $0<n<N$ for (\ref{eq:quadratur}), which 
corresponds to numerical quadrature with the trapezoidal rule, and note that (\ref{eq:equiv}) is fulfilled. 
Moreover, the matrices $\mathbf{M}_p,\mathbf{M}_m$, and $\mathbf{D}(\mathbf{m})$ are all diagonal and approximation order $h^2$ can be expected 
for sufficiently smooth solutions.
\end{example}

\begin{example}[Spectral method] \label{ex:spectral}
For $Q_h = P_{p-1}(0,1)$ and $V_h = P_p(0,1)\cap H^1$, Assumption \ref{ass:compat} holds as well. 
Now let $\xi_n$ and $\omega_n$, $0\le n \le p$, be the quadrature points and weights for the Gauss-Lobatto quadrature rule on $[0,1]$,
then also norm equivalence (\ref{eq:equiv}) is valid; cf. \cite{CanutoEtAl}, 
When choosing the Lagrange polynomials for the points $\{\xi_n\}_n$ as basis for $V_h$ and the Legendre polynomials as basis for $Q_h$, 
the matrices $\mathbf{M}_p,\mathbf{M}_m$, and $\mathbf{D}(\mathbf{m})$ are again diagonal. 
Here exponential convergence in $p$ can expected for smooth solutions \cite{CanutoEtAl}.
\end{example}

\begin{example}[Projection based model reduction] \label{ex:mor}
Let $Q_h,V_h,\omega_n,\xi_n$ be chosen as in Example 1 for small $h$ and let $Q_H\subset Q_h$, $V_H\subset V_h$ be constructed by a structure preserving model reduction approach \cite{BennerMehrmannSorensen05}, together with the modifications proposed in \cite{EggerKuglerEtAl18}. 
Then Assumption~\ref{ass:compat} holds and $\mathbf{M}_p$, $\mathbf{M}_m$ are diagonal for an appropriate choice of basis. 
Note that the evaluation of the nonlinear term $\mathbf{D}(\mathbf{m}(t))$ via (\ref{eq:quadratur}) still has the complexity of the high dimensional space $V_h$. Replacing (\ref{eq:quadratur}) by a quadrature rule with fewer quadrature points may be used to further reduce the complexity \cite{AntilEtAl13}. Let us note that uniform exponential stability can still be guaranteed for this complexity-reduction approach, as long as (\ref{eq:equiv}) is valid. 
\end{example}

\section{Numerical illustration} \label{sec:5}

Let us note that our results and methods of proof can be generalized almost verbatim to networks; see \cite{EggerKugler18}.
This will be illustrated now by some numerical tests, for which we utilize the network in Fig.~\ref{fig:network}. 
\begin{figure}[ht!]
\begin{center}
  \begin{tikzpicture}[scale=1.8]
  \node[circle,draw,inner sep=2pt] (v1) at (-1.87,0) {$v_1$};
  \node[circle,draw,inner sep=2pt] (v2) at (-0.87,0) {$v_2$};
  \node[circle,draw,inner sep=2pt] (v3) at (0,0.5) {$v_3$};
  \node[circle,draw,inner sep=2pt] (v4) at (0,-0.5) {$v_4$};
  \node[circle,draw,inner sep=2pt] (v5) at (0.87,0) {$v_5$};
  \node[circle,draw,inner sep=2pt] (v6) at (1.87,0) {$v_6$};
  \draw[->,thick,line width=1.5pt] (v1) -- node[above,sloped] {$e_1$} ++ (v2);
  \draw[->,thick,line width=1.5pt] (v2) -- node[above,sloped] {$e_2$} ++ (v3);
  \draw[->,thick,line width=1.5pt] (v2) -- node[above,sloped] {$e_3$} ++ (v4);
  \draw[->,thick,line width=1.5pt] (v3) -- node[above,sloped] {$e_4$} ++ (v4);
  \draw[->,thick,line width=1.5pt] (v3) -- node[above,sloped] {$e_5$} ++ (v5);
  \draw[->,thick,line width=1.5pt] (v4) -- node[above,sloped] {$e_6$} ++ (v5);
  \draw[->,thick,line width=1.5pt] (v5) -- node[above,sloped] {$e_7$} ++ (v6);
  \end{tikzpicture}
  \end{center}
    \vspace*{-1em}
 \caption{Network used for numerical tests.\label{fig:network}}
\end{figure}
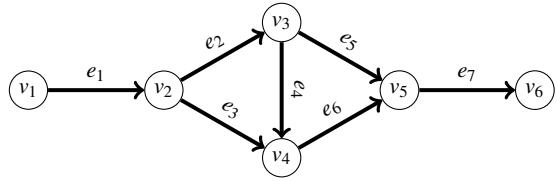
The topology of the network is represented by a directed graph $(\mathcal{V},\mathcal{E})$ with vertices $\mathcal{V} = \{v_1,\ldots,v_6\}$, divided into interior and exterior vertices $\mathcal{V}_{0}~=~\{v_2,\ldots,v_5\}$ and $\mathcal{V}_\partial = \{v_1,v_6\}$, and edges $\mathcal{E} = \{e_1,\ldots,e_7\}\subset\mathcal{V}\times\mathcal{V}$.
We denote by $\mathcal{E}(v) = \{e=(v,\cdot)\mbox{ or }e=(\cdot,v)\}$ the set of edges adjacent to the vertex $v$ and define 
$n^e(v) = -1$ for ingoing and $n^e(v) = 1$ for the outgoing pipes.

We then consider the following problem on the network: 
For every pipe $e\in\mathcal{E}$, the solution $( p^e, m^e)$ restricted to the pipe should satisfy (\ref{sys:inst1})--(\ref{sys:inst2}) 
with data $\bar{f},\bar{g} \equiv 0$, and damping function $d( m^e) = | m^e| m^e$. 
At the interior vertices of the network, the solution is required to satisfy the coupling conditions
\begin{eqnarray*}
  p^e(v,t) &=&  p^{e'}(v,t),\quad\mbox{ for all }e,e'\in\mathcal{E}(v),\;v\in\mathcal{V}_0,\;t>0,\\ 
 \sum_{e \in \mathcal{E}(v)}\!\!\!n^e(v) m^e(v,t) &=& 0,\qquad\quad\!\mbox{ for all }v\in\mathcal{V}_0,\;t>0,
\end{eqnarray*}
and we prescribe time dependent boundary conditions 
\begin{eqnarray*}
  p(v_1,t) = 90 + 10\max\{(1-t),0\} \quad\mbox{and}\quad  p(v_6,t) = 70.
\end{eqnarray*}
As initial conditions $p(0)$, $m(0)$, we choose the stationary solutions for the boundary data at time $t=0$.
The time discretization is chosen sufficiently accurate such that time integration errors can be neglected. 
For $T = 50$, we depict in Table~\ref{tab:exp} the exponential convergence of all methods. The POD method with $n_{sv}$ singular values is trained by an h-FEM method with $h = 10^{-3}$ and the correct boundary data. We choose $N$ Gauss-Lobatto points on each pipe such that (\ref{eq:equiv}) is satisfied. As predicted in Theorem~\ref{thm:decayh} the exponential decay is uniform in the discretization parameters. 

\setlength{\tabcolsep}{0.3em}
\begin{table}[ht!]
\renewcommand{\arraystretch}{1.2}
\small 
\begin{center}
\begin{tabular}{l||cccccc||c}
method$\setminus t^n$ & $0$ & $10$ & $20$ & $30$ & $40$ & $50$ & $\gamma$ \\
\hline
$\;$ Ex. \ref{ex:hfem}; $h = 0.2\;$& $99.136$ & $23.693$ & $6.943$ & $2.051$ & $0.607$ & $0.180$ & $0.122$ \\ 
$\;$ Ex. \ref{ex:hfem}; $h = 0.05\;$& $99.192$ & $23.709$ & $6.947$ & $2.052$ & $0.607$ & $0.180$ & $0.122$ \\ 
\hline
$\;$ Ex. \ref{ex:spectral}; $p = 3\;$& $99.196$ & $23.904$ & $7.005$ & $2.069$ & $0.613$ & $0.182$ & $0.122$ \\ 
$\;$ Ex. \ref{ex:spectral}; $p = 10\;$& $99.196$ & $23.710$ & $6.947$ & $2.052$ & $0.607$ & $0.180$ & $0.122$ \\ 
\hline
$\;$ Ex. \ref{ex:mor}; $n_{sv} = 2$ & $99.196$ & $23.850$ & $6.984$ & $2.062$ & $0.610$ & $0.181$ & $0.122$ \\ 
$\;$ Ex. \ref{ex:mor}; $n_{sv} = 10\;$& $99.196$ & $23.710$ & $6.947$ & $2.052$ & $0.607$ & $0.180$ & $0.122$ \\ 
\end{tabular}
\medskip
\caption{
Exponential convergence of $E( p_h(t)-\bar{p}_h, m_h(t)-\bar{m}_h)$ for the methods in Example \ref{ex:hfem}-\ref{ex:mor}.
}
\label{tab:exp}
\end{center}
\end{table}


\begin{acknowledgement}
The authors would like to gratefully acknowledge financial support by the German Research Foundation (DFG) via grants GSC~233, TRR~146, TRR~154, and Eg-331/1-1.
\end{acknowledgement}

\end{document}